\documentclass[11pt,english]{amsart}
\usepackage[latin9]{inputenc}
\usepackage{color}
\usepackage{babel}
\usepackage{verbatim}
\usepackage{amsthm}
\usepackage{amstext}
\usepackage{amssymb}
\usepackage{amsbsy}
\usepackage{ifpdf}
\usepackage[nobysame,abbrev,alphabetic]{amsrefs}

\usepackage{hyperref}%

\usepackage{enumerate}
\usepackage{amsmath}
\usepackage{amscd}
\usepackage{amsfonts}
\usepackage{graphicx}
\usepackage[all]{xy}
\usepackage{gensymb}
\usepackage{mathrsfs}

\makeatletter

\numberwithin{equation}{section}
\numberwithin{figure}{section}
\theoremstyle{plain}
\newtheorem{thm}{Theorem}[section]
  \theoremstyle{plain}
  \newtheorem{conjecture}[thm]{Conjecture}
  \theoremstyle{remark}
  \newtheorem{rem}[thm]{Remark}
  \theoremstyle{plain}
  \newtheorem{lem}[thm]{Lemma}
  \theoremstyle{plain}
  
  \theoremstyle{remark}
  \newtheorem*{claim*}{Claim}
  \theoremstyle{plain}
  \newtheorem{prop}[thm]{Proposition}

\theoremstyle{definition}

\theoremstyle{theorem}

\theoremstyle{definition}

\theoremstyle{definition}
\theoremstyle{definition}
\theoremstyle{definition}

\newcommand{\cO}{\mathcal{O}}
\newcommand{\cP}{\mathcal{P}}

\newcommand{\cU}{\mathcal{U}}

\newcommand{\cX}{\mathcal{X}}
\newcommand{\cY}{\mathcal{Y}}

\newcommand{\bC}{\mathbb{C}}
\newcommand{\bD}{\mathbb{D}}
\newcommand{\bG}{\mathbb{G}}
\newcommand{\bL}{\mathbb{L}}
\newcommand{\bM}{\mathbb{M}}
\newcommand{\bP}{\mathbb{P}}
\newcommand{\bR}{\mathbb{R}}
\newcommand{\bZ}{\mathbb{Z}}
\newcommand{\bQ}{\mathbb{Q}}
\newcommand{\bF}{\mathbb{F}}

\newcommand{\bN}{\mathbb{N}}
\newcommand{\bH}{\mathbb{H}}

\newcommand{\bS}{\mathbb{S}}

\newcommand{\bu}{\mathbf{u}}

\newcommand{\SL}{\operatorname{SL}}
\newcommand{\SO}{\operatorname{SO}}

\newcommand{\Mat}{\operatorname{Mat}}

\newcommand{\defi}{\overset{\on{def}}{=}}

\newcommand{\upa}{\pa{\mathbf{\frac{\mathbf{\mathbf{u}}}{\norm{\mathbf{\mathbf{u}}}}},[\Delta_{\mathbf{u}}]}}

\newcommand\norm[1]{\left\|#1\right\|}

\newcommand\set[1]{\left\{#1\right\}}
\newcommand\pa[1]{\left(#1\right)}

\newcommand\av[1]{\left|#1\right|}
\newcommand\on[1]{\operatorname{#1}}

\newcommand\mb[1]{\mathbf{#1}}

\newcommand{\lra}{\longrightarrow}
\newcommand{\ra}{\rightarrow}

\newcommand{\onto}{\xymatrix{\ar@{>>}[r]&}}
\newcommand{\da}[4]{\xymatrix{#1 \ar@<.5ex>[r]^{#2} \ar@<-.5ex>[r]_{#3} & #4}}

\newcommand{\slra}[1]{\stackrel{#1}{\lra}}

\begin{document}

\title[Integer points and their orthogonal grids]{Integer points on spheres and their orthogonal grids}

\begin{abstract}
 The set of primitive vectors on large spheres in the euclidean space of dimension~$d\geq3$
 equidistribute when projected on the unit sphere. We consider here a refinement of this problem concerning the direction
 of the vector together with the shape of the lattice 
  in its orthogonal complement. Using unipotent dynamics we obtained
 the desired equidistribution result in dimension $d\geq 6$ and in dimension $d=4,5$
 under a mild congruence condition on the square of the radius. The case of $d=3$
 is considered in a separate paper.
\end{abstract}

\author[M. Aka]{Menny Aka}

\address{Departement Mathematik\\
ETH Z\"urich\\
R\"amistrasse 101\\
8092 Zurich\\
Switzerland}

\email{menashe-hai.akka@math.ethz.ch}
\thanks{M.A.~acknowledges the support of ISEF, Advanced Research Grant
228304 from the ERC, and SNF Grant 200021-152819. }

\author[M. Einsiedler]{Manfred Einsiedler}

\address{Departement Mathematik\\
ETH Z\"urich\\
R\"amistrasse 101\\
8092 Zurich\\
Switzerland}

\email{manfred.einsiedler@math.ethz.ch}
\thanks{M.E.~acknowledges the support of the SNF Grant 200021-127145 and 200021-152819. }

\author[U. Shapira]{Uri Shapira}

\address{Department of Mathematics\\
Technion \\
Haifa \\
Israel }

\email{ushapira@gmail.com}
\thanks{U.S.~acknowledges the support of the Chaya fellowship and ISF grant 357/13.}


\maketitle

\section{Introduction}

Let $d\geq 3$ be a fixed integer. Let $\bZ_{{\rm prim}}^{d}$ be the
set of primitive vectors in $\bZ^{d}$. Set 
\[
\bS^{d-1}(D)\defi\set{v\in\bZ_{\textrm{prim}}^{d}:\norm{v}_{2}^{2}=D}=\bZ_{\rm prim}^d\cap \bigl(\sqrt{D}\bS^{d-1}\bigr),
\]
where $\bS^{d-1}\defi\set{x\in\bR^{d}:\norm{x}_{2}=\text{1}}$.  We would like to discuss
the simultaneous equidistribution of the direction~$\frac{v}{\sqrt{D}}\in\bS^{d-1}$
of the elements in~$\bS^{d-1}(D)$ and the {\em shape}~$[\Lambda_v]$ of the orthogonal lattice
\[
 \Lambda_v=\bZ^d\cap v^\perp.
\]

To make this more precise fix a copy of $\bR^{d-1}\defi\bR^{d-1}\times\set{0}$ in $\bR^{d}$
and choose for every~$v\in\bS^{d-1}(D)$ a rotation~$k_v\in\SO_d(\bR)$ with~$k_vv=\sqrt{D}e_d$
so that~$k_v\Lambda_v$ becomes a lattice in~$\bR^{d-1}$. Note that 
\begin{equation}
[\bZ^{d}:\pa{\bZ v\oplus\Lambda_{v}}]=D\label{eq:covolume-1}
\end{equation}
since primitivity of $v$ implies that the homomorphism $\bZ^{d}\ra\bZ$
defined by $u\mapsto(u,v)$ is surjective and $\bZ v\oplus\Lambda_{v}$
is the preimage of $D\bZ$. 

Therefore,~$k_v\Lambda_v$
is a lattice in~$\bR^{d-1}$ of covolume~$\sqrt{D}$. 
In order to normalize this covolume, we 
further multiply by the diagonal matrix 
$a_{v}={\rm diag}(D^{\frac{-1}{2(d-1)}},\ldots,D^{\frac{-1}{2(d-1)}},D^{\frac{1}{2}})$.
Note that the set of possible choice of~$k_v$ is precisely~$\SO_{d-1}(\bR)k_v$
and that $a_{v}$ commutes with $\SO_{d-1}(\bR)$. 
Recall that~$\SL_{d-1}(\bR)/\SL_{d-1}(\bZ)$ is identified
with the space of unimodular lattices in~$\bR^{d-1}$ so that we obtain an element
\[
 [\Lambda_v]\defi\SO_{d-1}(\bR)a_vk_v\Lambda_v\in
  \cX_{d-1}=\SO_{d-1}(\bR)\backslash \SL_{d-1}(\bR)/\SL_{d-1}(\bZ),
\]
which we refer to as the {\em shape of the lattice}~$\Lambda_v$.

It is possible to obtain still a bit more geometric information from the primitive 
vector~$v$ as follows.
Given~$v\in\bS^{d-1}(D)$ choose $w\in\bZ^{d}$ with $(w,v)=1$. If now~$v_1,\ldots,v_{d-1}$
is a~$\bZ$-basis of~$\Lambda_v$ we see that~$v_1,\ldots,v_{d-1},w$ is a~$\bZ$-basis
of~$\bZ^d$ and we may assume that $\det(v_{1},\cdots,v_{d-1},w)=1$. Let $g_{v}\in{\rm SL}_{d}(\bZ)$
denote the matrix whose columns are $v_{1},\ldots,v_{d-1},w$. 
Set~${\rm ASL}_{d-1}=\Bigl\{\begin{pmatrix}g&*\\0&1\end{pmatrix}\mid g\in\SL_{d-1}\Bigr\}$. The
set of possible choices for $g_{v}$ is the coset $g_{v}{\rm ASL}_{d-1}(\bZ)$.
We define a \emph{grid in $\bR^{d-1}$} to be a unimodular lattice~$\Lambda$ in $\bR^{d-1}$
together with a marked point on the $(d-1)$-dimensional torus~$\bR^{d-1}/\Lambda$. 
The space
${\rm ASL}_{d-1}(\bR)/{\rm ASL}_{d-1}(\bZ)$ is the moduli space of
grids in $\bR^{d-1}$.  Thus,
$a_{v}k_{v}g_{v}{\rm ASL}_{d-1}(\bZ)$ represents the grid consisting of
the rotated image of $\Lambda_{v}$ to $\bR^{d-1}$ together with the rotated
image of $w$ orthogonally projected into~$\bR^{d-1}$, and the well-defined double coset 
\begin{equation} \label{eq:Delta}
[\Delta_{v}]\defi{\rm SO}_{d-1}(\bR)a_{v}k_{v}g_{v}{\rm ASL}_{d-1}(\bZ)
\end{equation}
represents this grid up-to rotations of the hyperplane $\bR^{d-1}$.
Thus we obtain the element~$[\Delta_{v}]$ of the space 
\[
\cY_{d-1}\defi{\rm SO_{d-1}(\bR)\setminus ASL_{d-1}(\bR)/ASL_{d-1}(\bZ).}
\]
One should think about $[\Delta_{v}]$ as the shape of the orthogonal
lattice $\Lambda_{v}$ together with a point on the corresponding
$(d-1)$-dimensional torus which marks the position of orthogonal projection
of $w$ to the hyperplane containing $\Lambda_{v}$. 

Let $\tilde{\nu}_{D}$
denote the normalized counting measure on the set 
\[
\set{\pa{\frac{v}{\norm{v}},[\Delta_{v}]}:v\in\bS^{d-1}(D)}\subset\bS^{d-1}\times\cY_{d-1}.
\]
We are interested to find $A\subset\bN$ for which 
\begin{equation}
\tilde{\nu}_{D}\stackrel{\text{weak}^{*}}{\lra}m_{\bS^{d-1}}\otimes m_{\cY_{d-1}}\text{ as }D\ra\infty\text{ with }D\in A\label{eq:ASL2 convergences}
\end{equation}
where $m_{\bS^{d-1}}\otimes m_{\cY_{d-1}}$ is the product of the
natural uniform measures on $\bS^{d-1}$ and $\cY_{d-1}$. We propose
the following conjecture as a generalization of Linnik's Problem on
spheres (see e.g.\ \cite[\S 1]{MV2006ICM} for a survey):
\begin{conjecture}
\label{Conj:ASLd conjecture}The convergence in (\ref{eq:ASL2 convergences})
holds for the subset $A=\bN$ if~$d>4$, holds for~$A=\bN\setminus(8\bN)$ if~$d=4$,
and for the subset
\[ 
 A=\bigl\{D\geq 1\mid D \mbox{ is not congruent to~$0,4,7$ modulo }8\bigr\}
\] 
if~$d=3$.
\end{conjecture}

By a theorem of Legendre the restriction to the proper subset of~$\bN$ as in the above conjecture for~$d=3$
is equivalent to~$\bS^{2}(D)$ being nonempty, and hence necessary. A similar statement holds for~$d=4$.

In a separate paper \cite{AES3} we obtain for the case~$d=3$ some partial results towards this conjecture.
However, for~$d>3$ we can give much stronger results using the techniques presented here.
For an odd rational prime $p$ let $\bD(p)=\set{D:p\nmid D}$. The
main result of this paper is the following:

\begin{thm}
\label{thm:main theorem}Conjecture \ref{Conj:ASLd conjecture} is
true for $d>5$. For $d=5$ (resp.~$d=4$), the convergence in (\ref{eq:ASL2 convergences})
holds for the subset $A=\bD(p)$ (resp.~$A=\bD(p)\setminus(8\bN)$) where $p$ is any fixed odd prime.
\end{thm}

The reader may find it interesting to work out the meaning of Theorem~\ref{thm:main theorem} when one restricts the equidistribution in (\ref{eq:ASL2 convergences}) to open sets of the form $\bS^{d-1}\times U$ and $V\times \cY_{d-1}$ where $U$ is an arbitrary open set in $\cY_{d-1}$ (say, the neighbourhood of a specific lattice) and $V$ is an arbitrary open set in $\bS^{d-1}$ (say, an open ball centred at some direction on the sphere). 

Theorem \ref{thm:main theorem} will be proven using the theorem of Mozes and Shah \cite{MS95} concerning limits of algebraic
measures with unipotent flows acting ergodically. More precisely we will need a~$p$-adic analogue
of this result, which has been given more recently by Gorodnik and Oh \cite{GorOh2011}. 
In particular we note that Theorem~\ref{thm:main theorem} should therefore be considered a corollary of
the measure classification theorems for unipotent flows on~$S$-arithmetic quotients
(see \cite{RAT95} and \cite{MT94}). 

As explained in Lemma \ref{lem:Isotropy of Hv },
the congruence condition $D\in \bD(p)$ is a splitting condition which
enables us to use the existing theory of unipotent dynamics.
It is possible to remove this splitting condition for~$d=4,5$
by giving effective dynamical arguments in the spirit
of~\cite{EMMV} (see also~\cite{EMV}), but this result is not general
enough for that purpose. In~\cite{ERW2015} Ren\'e R\"uhr, Philipp Wirth, and the second
named author use the methods of~\cite{EMMV}
to obtain equidistribution on~$ \bS^{d-1}\times\cX_{d-1}$  for~$d=4,5$
without imposing any congruence condition. 
We note however, that the case~$d=3$ remains open (apart from the partial
results in \cite{AES3} that concerns itself only with the problem on~$\bS^2\times\cX_2$
and involves some stronger congruence conditions).

Our interest in this problem arose through the work of W.~Schmidt~\cite{WSchmidt-sublattices},
J.~Marklof~\cite{Marklof-Frobenius} (see also~\cite{EMSS}). However, as Peter Sarnak and Ruixiang Zhang
pointed out to us, Maass \cite{Maass1956} already asked similar questions in 1956 (see
also~\cite{Maass1959}). The generalisation to our joint equidistribution problem seems to be new. Thus, one may view the above question as the common refinement
of Linnik's problem and the question of Maass. 
\subsection*{Acknowledgements}
We would like to thank Andreas Str{\"o}mbergsson and the anonymous referee for suggestions and comments. While working on this project the authors visited the Israel Institute of Advanced Studies (IIAS) at the Hebrew University and its hospitality is deeply appreciated.

\section{Notation and organization of the paper}

A sequence of probability measures $\mu_{n}$ on a metric space
$X$ is said to equidistribute to a probability measure $\mu$ if
$\mu_{n}$ converge to $\mu$ in the weak$^{*}$ topology on the space
of probability measures on $X$. A probability measure $\mu$ is called
a weak$^{*}$\emph{ limit }of a sequence of measures $\mu_{n}$ if
there exists a subsequence $(n_{k})$ such that $\mu_{n_{k}}$ equidistribute
to $\mu$ as $k\ra\infty$. For a probability measure $\mu$ and a
measurable set $A$ of positive measure, the restriction $\mu|_{A}$
of $\mu$ to $A$ is defined by $\mu|_{A}(B)=\frac{1}{\mu(A)}\mu(A\cap B)$
for any measurable set $B$.

Recall that a discrete subgroup~$\Lambda<L$ is called a lattice if~$L/\Lambda$
admits an~$L$-invariant probability measure.
Given a locally compact group $L$ and a subgroup $M<L$ such that
$L/M$ admits an $L$-invariant probability measure, it is unique,
we denote it by $\mu_{L/M}$, and call it the \emph{uniform measure}
or the \emph{Haar measure} on $L/M$. If furthermore,~$K<L$ is
a compact subgroup then there is a natural quotient map~$L/M\to K\backslash L/M$,
and the uniform measure on~$K\backslash L/M$
is by definition the push forward of the Haar measure on~$L/M$. 

We recall that the Haar measure
on a finite volume orbit $Hg\Gamma$ is the push-forward of the
uniform measure on $H/\pa{H\cap g\Gamma g^{-1}}$. Note that a twisted orbit
of the form $gH\Gamma$ can be thought of as an orbit for the
subgroup~$gHg^{-1}$ since $gH\Gamma=gHg^{-1}g\Gamma$. 

For a finite set $S$ of valuations on $\bQ$ we set $\bQ_{S}=\prod_{v\in S}\bQ_{v}$
and $\bZ^S=\bZ\bigl[\bigl\{\frac{1}{p}:p\in S\setminus\{\infty\}\bigr\}\bigr]$. For a prime number $p$,
$\bZ_{p}$ denotes the ring of $p$-adic integers (so with this notation $\bZ^{\set{p}}\cap\bZ_{p}=\bZ$).
As usual we will embedd~$\bZ^S$ diagonally into~$\bQ_S$, where~$q\in\bZ^S$ is mapped to~$(q,\ldots,q)\in\bQ_S$. When 
$\infty\in S$, the group $\bZ^S$ is a discrete and cocompact subgroup of~$\bQ_S$.

For an algebraic group $\bP$ we write $\bP_{S}\defi\bP(\bQ_{S})$.
For a semisimple algebraic $\bQ$-group $\bP$ we let $\pi_{\bP}:\tilde{\bP}\ra\bP$
be the simply connected covering map over $\bQ$,  which is unique
up-to $\bQ$-isomorphism (see \cite[Thereom 2.6]{PR94} for details). We denote by $\bP_{S}^{+}$ the image of
$\tilde{\bP}_{S}$ under $\pi_{\bP}$. In some cases (which will be relevant to us) $\bP_{S}^{+}$ agrees with the group generated by one-parameter unipotent subgroups of $\bP_S$ (see \cite[Cor. 6.5]{BT73} and Section \ref{sec:spinor} for more information).
We also recall that~$\bP(\bZ^S)$
is a lattice in~$\bP(\bQ_S)$ if~$\infty\in S$ and~$\bP$ is semisimple and also if~$\bP={\rm ASL}_{d-1}$.
As we will see later the subgroup~$\bP_S^+<\bP_S$ plays an important role
in some of the ergodic theorems on~$\bP(\bQ_S)/\bP(\bZ^S)$ that we will use.

The letter $e$ will always denote the
identity element of a group, and we will sometimes use subscripts to indicate the corresponding group,
e.g.~we may write~$e_\infty\in\bP(\bR), e_p\in\bP(\bQ_p)$, or~$e_f\in\bP(\prod_{p\in S\setminus\{\infty\}}\bQ_p)$.

This paper is organised as follows: The desired equidistribution \eqref{eq:ASL2 convergences} follows from an equidistribution of "joined" orbits on a product of $S$-adic homogeneous spaces, which is proved in \S \ref{sec:orbits and duality}--\ref{sec:proof} using unipotent dynamics. The translation between the result of \S \ref{sec:orbits and duality} to \eqref{eq:ASL2 convergences} is stated in \S \ref{sec:equivalence relation} and proved in \S \ref{sec:translating}.

\section{Equidistribution of joined $S$-adic orbits\label{sec:orbits and duality}}

Fix a finite set~$S\ni\infty$ of valuations of~$\bQ$.
We define the algebraic groups $\bG_{1}={\rm SO_{d}},\bG_{2}={\rm ASL_{d-1}}$, $\overline{\bG}_{2}={\rm SL_{d-1}}$,
$\bG=\bG_{1}\times\bG_{2}$, and $\overline{\bG}=\bG_{1}\times\overline{\bG_{2}}$.
Consider the homogeneous spaces $\cY^{S}\defi\bG_{S}/\bG(\bZ^S)$ and
$\cX^{S}\defi\overline{\bG}_{S}/\overline{\bG}(\bZ^S)$.
We write $\pi^{S}:\cY^{S}\ra\cX^{S}$ for
the map induced by the natural projection $\rho:{\rm ASL}_{d-1}\ra{\rm SL}_{d-1}$.
Finally, let $\cY_{i}^{S}=\bG_{i,S}/\bG_{i}(\bZ^S)$.

For $v\in \bZ_{\textrm{prim}}^{d}$ we set $\bH_{v}\defi{\rm Stab}_{\bG_{1}}(v)$ under the natural action.
The group $\bH_{v}$ is defined over $\bZ\subset\bQ$ as $v\in\bZ^{d}$.
Let us mention at this point that we will prove Theorem~\ref{thm:main theorem}
by studying the dynamics and the orbits of the stabilizer~$\bH_v$ of~$v$
which in particular will allow us to conclude that there are many primitive
vectors in~$\bS^{d-1}(\|v\|^2)$ if~$D=\|v\|^2$ is sufficiently large.

\subsection{Dynamically producing primitive points on a sphere} 

Let us try to outline with a minimum of technology the key idea 
of why near density on the above homogeneous space (below proven by dynamical methods) 
can give us near density of primitive points on large spheres.
This idea is by no means new and to some extend implicitly 
appears already in the work of Linnik, and is explicitly 
used e.g.\ in the work of Ellenberg and Venkatesh \cite{EV08}. 

Let~$v\in\bS^{d-1}(D)$ be a large primitive point. Suppose that we already know for this vector
and some fixed odd prime~$p$  that the orbit~$\bH_v(\bQ_S)\bG_1(\bZ^S)$
is quite dense in~$\cY^S_1$ for~$S=\{\infty,p\}$. Given an arbitrary~$g_\infty\in\bG_1(\bR)$
we can then find some~$(h_\infty,h_p)\in\bH_v(\bQ_S)$ that gives us the approximate identity
\[
(h_\infty,h_p)\bG_1(\bZ^S)\approx(g_\infty,e)\bG_1(\bZ^S).
\]
Going back to the group we see that there exists some
diagonally embedded lattice element~$(\gamma,\gamma)\in\bG_1(\bZ^S)$
such that
\[
(h_\infty,h_p)(\gamma,\gamma) \approx(g_\infty,e) .
\]
We claim that~$w=\gamma^{-1}v\in\bZ^d$ is an integer vector, on the same sphere,
and in direction close to the arbitrary direction~$g_\infty^{-1}v$.
Indeed,~$\gamma^{-1}v=\gamma^{-1}h_\infty^{-1}v\approx g_\infty^{-1}v$ by using
the real component of the above approximate identity and~$\Vert\gamma^{-1}v\Vert=\Vert v\Vert=\sqrt{D}$
since~$\gamma\in{\rm SO}_d(\bZ[\frac1p])$.  
Moreover,~$w\in(\bZ[\frac1p])^d$ since~$\gamma$ has entries in~$\bZ[\frac1p]$
and~$\gamma^{-1}v=\gamma^{-1}h_p^{-1}v\approx v$ with respect to the topology in~$\bQ_p^d$ which implies~$\gamma^{-1}v\in\bZ^d$ as required. 

\subsection{Towards equidistribution}
We continue to setup some notation for the joint equidistribution.
We set ${\rm SO}_{d-1}(\bR)={\rm Stab}_{\bG_{1}}(e_{d})(\bR)$ and note
that $k_{v}^{-1}{\rm Stab}_{\bG_{1}}(e_{d})(\bR)k_{v}=k_{v}^{-1}{\rm SO}_{d-1}(\bR)k_{v}=\bH_{v}(\bR)$.
Consider the diagonally embedded algebraic subgroup $\bL_{v}<\bG$ defined
by 
\begin{equation}\label{def-lv}
\bL_{v}(R)\defi\set{\pa{h,g_{v}^{-1}hg_{v}}:h\in\bH_{v}(R)}
\end{equation}
for any ring $R$. As $g_{v}\in{\rm SL}_{d}(\bZ)$, the group $\bL_{v}$
is also defined over $\bZ\subset\bQ$. 


For $v\in\bS^{d-1}(D)$ let $\theta_{v}=a_{v}k_{v}g_{v}\in \bG_2(\bR)$ 
and consider the orbit
\begin{equation}\label{defOvS}
\mathbf{O}_{v,S}\defi(k_{v},e_{f},\theta_{v},e_{f})\bL_{v,S}\bG(\bZ^S)\subset\cY^{S}
\end{equation}
As $\bL_{v}$ is $\bQ$-anisotropic (e.g.$\ $because $\bL_{v}(\bR)$
is compact), the Borel Harish-Chandra Theorem (see e.g.\ \cite[Theorem I.3.2.4]{MR91}) implies
that this is a compact orbit. Set
 $\mu_{v,S}$ 
to be the Haar measure on this orbit  
and finally define $\mu_{S}\defi\mu_{\bG_{S}/\bG(\bZ^S)}$. 

\begin{thm}
\label{thm:Hom spaces equi}Let $p$ be a fixed odd prime and set
$S=\{\infty,p\}$. For $d>5$ and for any sequence $\set{v_{n}}\subset\bZ_{\textrm{prim}}^{d}$
with~$\|v_n\|\to\infty$ as~$n\to\infty$
we have that $\mu_{v_{n},S}$ converge in the weak$^{*}$ topology to
$\mu_{S}$. \\
The same conclusion holds for $d=4\text{ or }5$ when $\set{v_{n}}$
is a sequence of primitive vectors with~$\|v_n\|\to\infty$ as~$n\to\infty$
and $\norm{v_{n}}_{2}^{2}\in\bD(p)$
for any $n\in\bN$.
\end{thm}
This section contains some preparations for the proof of Theorem \ref{thm:Hom spaces equi} which is  proven in Section \ref{sec:proof}.

\subsection{Facts on  quadratic forms }
Let $Q_0$ denote the quadratic form $\sum_{i=1}^{d}x_{i}^{2}$.
Fix a vector $v\in\bS^{d-1}(D)$ and a rational matrix $\gamma\in{\rm SL_{d-1}}(\bQ)$.
Fix a choice of $g_{v}$ and consider the following quadratic map
\[
\phi_{v}^{\gamma}:\bQ^{d-1}\ra\bQ,\, u\mapsto(Q_0\circ g_{v}\circ\gamma)(u).
\]
As before we identify~$\bQ^{d-1}$ with~$\bQ^{d-1}\times\{0\}\subset\bQ^d$ so that~$g_v(\gamma(u))\in\bQ^d$
is well-defined. 
We set $\phi_{v}\defi\phi_{v}^{e}$. For a quadratic map $\phi$ let
$B_{\phi}$ be the associated bilinear form 
\[ 
 B_{\phi}(u_{1},u_{2})=\frac{1}{2}\big ( \phi(u_{1}+u_{2})-\phi(u_{1})-\phi(u_{2})\big ).
\]
Finally, the determinant of $\phi$ with respect to $b_{1},\ldots,b_{d-1}$
is $\det M_{\phi}$ where $M_{\phi}=\pa{B_{\phi}(b_{i},b_{j})}_{1\leq i,j\leq d-1}$.
When $b_{1},\ldots,b_{d-1}$ is a basis for $\bZ^{d-1}\subset\bQ^{d-1}$
and $B_{\phi}(b_{i},b_{j})\in\bZ$ for all $1\leq i,j\leq d-1$, the determinant is a well-defined
integer which does not depend of the choice of the basis. This is
the case for $\phi_{v}$ with the standard basis and the choice of
$g_{v}$ merely changes the basis, so does not influence the value
of the determinant. 

\begin{lem}
\label{lem:determinant and primitivity}
For any $v\in\bS^{d-1}(D)$
we have $B_{\phi_v}(\bZ,\bZ)\subset\bZ$ and $\det(\phi_{v})=D$. 
Moreover, there exist $u_{1},u_{2}\in\bZ^{d-1}$
such that $B_{\phi_{v}}(u_{1},u_{2})=1$.
In other words, the companion matrix $M_{\phi_v}$ is a primitive integer matrix.
\end{lem}

\begin{proof}
Using Equation (\ref{eq:covolume-1}) we see that the determinant of $Q_0$
with respect $v_{1},\ldots,v_{d-1},v$ is $D^{2}$. But $M_{Q_0}$ with
respect this basis, is a block matrix having a $d-1$-block whose
determinant is $\det(\phi_{v})$ and a $1$-block whose value is $(v,v)=D$.
Thus the first assertion follows.

Let $v_{1},\ldots,v_{d-1},w$ be a basis chosen as in the introduction.
Since~$d\geq 3$ we can assume without loss of generality that $v_{1}$ is in $v^{\perp}\cap w^{\perp}$.
It is enough to show the second assertion while considering the map
$\phi_{v}$ with this choice of a basis as the columns of the matrix
$g_{v}$. As $v_{1}$ is primitive we can find $u\in\bZ^{d}$ with
$(u,v_{1})=1$. As $(w,v_{1})=0$ we can add to $u$ multiples of
$w$, so we can assume (as $v_{1},\ldots,v_{d-1},w$ is a basis for
$\bZ^{d}$) that $u=\sum_{i=1}^{d-1}a_{i}v_{i},\: a_{i}\in\bZ$. This
implies that $\sum_{i=1}^{d-1}a_{i}(v_{i},v_{1})=1$ showing $B_{\phi_{v}}(e_{1},(a_{1},\ldots,a_{d-1}))=1$.
\end{proof}

Define $\bH_{\phi}={\rm SO}(\phi)< {\rm SL}_{d-1}$
by $\set{T\in{\rm SL}_{d-1}:\phi\circ T=\phi}$. Recall that
$\rho:{\rm ASL}_{d-1}\ra{\rm SL}_{d-1}$ denotes
the natural projection. Following the definitions, we have that 
\begin{equation}
\bH_{\phi_{v}^{\gamma}}=
 \gamma^{-1}\bH_{\phi_{v}}\gamma=
  \gamma^{-1}\rho(g_{v}^{-1}\bH_{v}g_{v})\gamma.\label{eq:isotropy group equal to sec. factor}
\end{equation}

\begin{lem}
\label{lem:isotropy subgroup determine}Let $\phi,\phi':\bQ^{d-1}\ra\bQ$
be two non-degenerate quadratic maps and assume that $\bH_{\phi}=\bH_{\phi'}$ as
$\bQ$-algebraic subgroups of ${\rm SL}_{d-1}$. Then there exists
$r\in\bQ^{\times}$ such that $\phi=r\phi'$.\end{lem}
\begin{proof}
It is enough to prove this statement over $\bC$. Thus, we can assume
that $M_{\phi}$ is the identity matrix. We need to show that $M_{\phi'}$
is a scalar matrix. Fix $1\leq i<j\leq d-1$ and let $M_{\phi}^{ij}$
be the 2 by 2 matrix whose entries are the $ii,ij,ji,jj$ entries
of $M_{\phi}$ and similarly for $M_{\phi'}^{ij}$. Acting on matrices
with $A.M=A^{t}MA$, $M_{\phi}^{ij}$ is preserved by ${\rm SO_{2}(\bC)}$,
and so is $M_{\phi'}^{ij}$ by our assumption. A direct calculation
show that this implies that $M_{\phi'}^{ij}={\rm diag}(r,r)$ for
some $r\neq0$. Applying this argument for all possible $i\neq j$
implies the claim.\end{proof}


\begin{lem}
\label{lem:Isotropy of Hv }Let $v\in\bS^{d-1}(D)$ and recall that  $p\neq 2$. Then the Lie algebra of  $\bH_{v}$ (resp.\ $\rho(g_{v}^{-1}\bH_{v}g_{v})$) is a maximal  semisimple Lie sub-algebra of the Lie algebra of $\bG_{1}$ (resp.\ $\overline{\bG_2}$) with finite centralizer. For $d>5$ the group $\bH_{v}(\bQ_p)$  is $\bQ_p$-simple and isotropic. The same holds for $d=4$ when $D\in\bD(p)$. For $d=5$ under the assumption $D\in\bD(p)$, $\bH_{v}(\bQ_p)$  is semisimple with each $\bQ_p$-simple almost direct factor being isotropic.\end{lem}

\begin{proof}
Maximality goes back to a classification made in 1952 by Dynkin \cite{dynkin1952maximal} whose english translation may be found at~\cite{dynkin2000maximal}. 
As being isotropic is preserved by conjugation by $g_v$,  it is enough to prove the statements for $\bH_{v}$. The group $\bH_{v}$ is naturally the orthogonal group of the $d-1$ dimensional quadratic lattice $Q_v\defi(Q_0,\Lambda_v)$. For $d-1\geq 5$,  $Q_v$ is automatically isotropic over $\bQ_p$.  For $d=4,5$ we have seen in Lemma~\ref{lem:determinant and primitivity} that  $Q_v$ has discriminant $D$. Denote the Hasse invariant of $Q_v$ by $S(Q_v)$.  Note that for $p\neq 2$ the congruence condition $p\nmid D$ implies that $S(Q_v)=1$ (use \cite[Theorem~3.3.1(d)]{KIT99}). This, in turn, implies that $Q_v$ is isotropic (use \cite[Theorem~3.5.1]{KIT99}).

Note that for $d>5$ or $d=4$ the algebraic group $\bH_v$ is absolutely simple so the lemma follows in this case. 
The case $d=5$ requires special attention as the Lie algebra ${\mathfrak s\mathfrak o}_4 $ is isomorphic to ${\mathfrak s\mathfrak l}_2\times {\mathfrak s\mathfrak l}_2$ over the algebraic closure. As explained above, under the assumption that $p\nmid D$  the algebraic group $\bH_v$ is $\bQ_p$-isotropic. Note that the $\bQ_p$-rank of ${\rm Lie}(\bH_v)$ is equal to the isotropy rank of $Q_v$. Thus this rank is either 1 or 2.

If the isotropy rank is 2, then $\bH_v(\bQ_p)$ is conjugate to the split special orthogonal group $\SO(xy+zw)$ over $\bQ_p$ (use \cite[\S 2.2]{cassels:RQF}). In this case  ${\rm Lie}(\bH_v)$ is $\bQ_p$-isomorphic to the Lie algebra ${\mathfrak s\mathfrak l}_2\times {\mathfrak s\mathfrak l}_2$ over $\bQ_p$ (as
may be seen by studying the representation of $\SL_2(\bQ_p)\times\SL_2(\bQ_p)$ by multiplication
on both sides on $\Mat_{22}(\bQ_p)$ together with the quadratic form $\det$). Therefore each of its almost direct simple factors is $\bQ_p$-isotropic.

If the isotropy rank of $Q_v$ is $1$ then $\bH_v(\bQ_p)$ is conjugate to the orthogonal group of the quaternary quadratic form $xy-(z^2-\eta w^2) $ for some $\eta\in \bQ_p^\times\setminus \pa{\bQ_p^\times}^2$. This implies that $\bH_v(\bQ_p)$ is simple as its Lie algebra is $\bQ_p$-isomorphic the Lie algebra of ${\rm Res}_{\bQ_p(\sqrt{\eta})/\bQ_p}\SL_2$.
If fact, let $\bF=\bQ_p(\sqrt{\eta})$. Then
the isomorphism\footnote{This is the $p$-adic analog of the isogeny from $\SL_2(\bC)$ to $\SO^+(3,1)(\bR)$.} arises from the representation of $\SL_2(\bF)$
on 
\[
 V=\{A\in\operatorname{Mat}_{22}(\bF)\mid A^*=A\} 
\]
(with $A^*=\overline{A^T}$ denoting the Galois conjugate of the transpose matrix) defined by $g.A=g A g^*$
for all $g\in\SL_2(\bF)$ and $A\in V$. Since $\det g=1$ this preserves the quadratic form $\det$
on $V$, and since 
\[
 \det\begin{pmatrix}x&z+\sqrt{\eta}w\\ z-\sqrt{\eta}w &y\end{pmatrix} 
  =xy-(z^2-\eta w^2)
 \]
the isomorphism of the Lie algebras follows.

Finally, the finiteness of the centralizer follows as otherwise its product with $\bH_v$ would give a proper subgroup containing the maximal group $\bH_v$. 
\end{proof}

\subsection{Limits of algebraic measures}

Let $\bG\subset\SL_{k}$ (for some integer~$k$) be a connected semisimple $\bQ$-group, $S$ a finite set
of valuations containing all the valuations for which $\bG(\bQ_{v})$
is compact,~$\bG_S=\bG(\bQ_S)$ and $\Gamma$ a finite-index subgroup of $\bG(\bZ^S)=\bG(\bQ_S)\cap\SL_k(\bZ^S)$.
Let $X_{S}\defi \bG_{S}/\Gamma$ and let $\cP(X_{S})$ denote the space of probability
measures on $X_{S}$.

Mozes and Shah showed in \cite{MS95} that limits of algebraic probability measure
are again algebraic if some unipotent flows act ergodically for each of the measures
in the sequence. We are going to use the following analogue for~$S$-arithmetic quotients
obtained by Gorodnik and Oh, which we state here in a slightly simplified
version. 

\begin{thm}[{\cite[Theorem 4.6]{GorOh2011}}]
\label{thm:gorod-oh}Let ${\mathbf L}_{i}$ be a sequence of connected
semisimple $\bQ$-subgroups of $\bG$ and assume that there exists $p\in S$
such that for any ${\mathbf L}_{i}$ and any non-trivial normal subgroup ${\mathbf N}<{\mathbf L}_{i}$ which is defined over $\bQ$,
${\mathbf N}(\bQ_{p})$ is non-compact. Let ${g_{i}}$ be a sequence of
elements of $\bG_{S}$ and set $\nu_{i}\defi\mu_{g_{i}{\mathbf L}_{i,S}^{+}\Gamma}$.
If the centralizers of all ${\mathbf L}_{i}$ are $\bQ$-anisotropic, then 
$\{\nu_{1},\nu_2,\ldots\}$ is relatively compact in ${\mathcal P}(X_{S})$. 
Assume that $\nu_{i}$ weakly converge to $\nu$ in ${\mathcal P}(X_{S})$, then the
following statements hold:
\begin{enumerate}
\item There exists a Zariski connected $\bQ$-algebraic subgroup $\bM$ of $\bG$ such that $\nu=\mu_{gM\Gamma}$
where $M$ is a closed finite-index subgroup of $\bM_{S}$ and $g\in \bG_{S}$.
If the centralizers of all~$\mathbf L_i$ are~$\bQ$-anisotropic, then~$\bM$ is semi-simple.
\item There exists a sequence ${\gamma_{i}}\subset\Gamma$ such that
for all $i$ sufficiently large we have $\gamma_{i}^{-1}{\mathbf L}_{i}\gamma_{i}\subset\bM$.
\item There exists a sequence $h_{i}\in{\mathbf L}_{i,S}^{+}$ such that $g_{i}h_{i}\gamma_{i}$
converges to $g$ as $i\ra\infty$.
\end{enumerate}
\end{thm}

Theorem \ref{thm:gorod-oh} deals with orbits of groups of the form ${\mathbf L}_S^+$ whereas Theorem~\ref{thm:Hom spaces equi} deals with orbits of groups of the form ${\mathbf L}_S$. Therefore we collect some information on the quotient ${\mathbf L}_S/{\mathbf L}_S^+$.
\subsection{The spinor norm} \label{sec:spinor}
Throughout this section, we fix a regular rational quadratic form $(\bQ^d,q)$. For any field $\bQ\subset F$ we have the spinor-norm map 
$$ \phi=\phi_{\SO_q}:\SO_q(F)\ra F^\times/\pa{F^\times}^2$$
defined as follows: Any element $g\in\SO_q(F)$ can be written as $g=\tau_{v_1}\cdots\tau_{v_r}$ where $v_1,\ldots,v_r\in F^d$, $r$ is even and $\tau_v$ is  a reflection (also called symmetry) with respect to $v$ (see \cite[Ch.2 \S 4]{cassels:RQF}). The spinor-norm is defined as $\phi(g)=\prod_{i=1}^r q(v_i)\pa{F^\times}^2$.

\begin{lem}\label{lem:spin}
Let $F=\bQ_p$ for an odd prime $p$ or $F=\bR$.
\begin{enumerate}
\item\label{spin seq} We have the short exact sequence 
$$ {\rm Spin}_q(F)\slra{\pi_{\SO_q}}\SO_q(F)\slra{\phi} F^\times/\pa{F^\times}^2,$$
where ${\rm Spin}_q$ denotes the simply-connected cover of $\SO_q$. 
\item\label{spin real} In particular $\pi_{\SO_q}$ is onto when  $F=\bR$ and $q$ is positive-definite.
\item \label{spin onto}For $\bF=\bQ_p$ for odd $p$, and $d\geq 3$, the map $\phi$ is onto $\bQ_p^\times/\pa{\bQ_p^\times}^2$. The latter is a group of size four whose elements are denoted as $\set{1,r,p,rp}$.
\item \label{spin unipotents}If $q$ is isotropic then the group $\SO_q(F)^+$  is the group generated by one-dimensional unipotent subgroups.
\item \label{spin restriction} We have $\phi_{\bH_v}=\phi_{\bG_1}|_{\bH_v(F)}$.
\end{enumerate}
\end{lem}
\begin{proof}
Part \ref{spin seq} is proved in \cite[Ch.\ 10 Thm.\  3.3]{cassels:RQF} and Part \ref{spin real} readily follows from it.
Parts \ref{spin onto} and \ref{spin unipotents} are proved in \cite[Lemma 1]{EV08}. Note that Part \ref{spin onto} follows easily when $q$ is isotropic (as $q$ achieves any value in $F$), which is the case of interest for our application.
Part \ref{spin restriction} follows from the definition of $\phi$ and the fact that $\bH_v$ is the orthogonal group of the restriction of $\sum_{i=1}^d x_i^2$ to $v^\perp$. 
\end{proof}


With the following lemma we reduce Theorem \ref{thm:Hom spaces equi} to a statement (see \eqref{hom equi spin}) which is approachable by Theorem \ref{thm:gorod-oh}.
\begin{lem}\label{lem:finite index normal cover}
Let $H',H,G$ be locally compact groups with $H'\lhd H<G$, $[H:H']=k$ and $h_1,\ldots, h_k$ a complete list of representatives of $H'$ cosets in $H$. Let $\Gamma$ be a lattice in $G$ and consider a finite volume orbit $Hg\Gamma$ where $g\in G$. Then 
$$\mu_{Hg\Gamma}=\frac 1k \sum_{i=1}^k\mu_{H'h_ig\Gamma}=\frac 1k \sum_{i=1}^k\mu_{h_iH'g\Gamma}.$$
\end{lem}
\begin{proof}
Normality implies the second equality. The last two expressions define a probability measure which is supported on $Hg\Gamma$. This measure  is  $H$-invariant as it is clearly $H'$-invariant and translating by $h_j$ only permutes the summands. The lemma follows, since $\mu_{Hg\Gamma}$ is the unique measure with these properties.
\end{proof}
Fix $v\in\bS^{d-1}(D)$. We decompose the orbit in \eqref{defOvS} into orbits involving the group  ${\bL}_{v,S}^{+}$. To this end, recall the notation\footnote{We 
write~$r$ for the image of an element~$\bZ_p^\times$ that is not a square in~$\bZ_p$ and by a slight abuse of notation~$p$ for the image of the uniformizer~$p\in\bQ_p$ in~$\bQ_p^\times/(\bQ_p^\times)^2$.}
$\bQ_p^\times/(\bQ_p^\times)^2=\set{1,r,p,rp}$ and 
fix $S=\set{\infty,p}$ for an odd prime $p$. Let $h_1^v,h_{r}^v,h_{p}^v,h_{rp}^v\in \bH_v(\bQ_p)$ denote elements with $\phi(h_i^v)=i$. Similarly, let  $g_1,g_{r},g_{p},g_{rp}\in \bG_1(\bQ_p)$ denote elements with $\phi(g_i)=i$. Finally, let  $\bG_{S}^{+}$ denote $\pa{\bG_1}_{S}^{+}\times \pa{\bG_2}_S$. For $i\in \set{1,r,p,rp}$ define    
$$\cY^{S,i}\defi \bG_{S}^{+}(e_\infty,g_i,e_\infty,e_p)\bG(\bZ^S),\, \mu_{S}^i\defi \mu_{\cY^{S,i}}$$ 

$$\mathbf{O}_{v,S}^i\defi(k_{v},h_i^v,\theta_{v},g_v^{-1}h_i^vg_v)\bL_{v,S}^+\bG(\bZ^S)\subset\cY^{S,i},\, \mu_{v,S}^i\defi\mu_{\mathbf{O}_{v,S}^i}$$
where the last containment follows from parts (\ref{spin restriction}) and (\ref{spin real}) of Lemma  \ref{lem:spin}.

Applying Lemma \ref{lem:finite index normal cover} with $\bL_{v,S}^+ \lhd \bL_{v,S}$ (resp.\  with $\bG_{S}^+ \lhd \bG_{S}$) we see that $\mu_{v,S}^i=\frac 14 \sum_{i\in \set{1,r,p,rp}}\mu_{v,S}^i$ and resp.\ that 
$\mu_{S}^i=\frac 14 \sum_{i\in \set{1,r,p,rp}}\mu_{S}^i$. Therefore Theorem \ref{thm:Hom spaces equi} follows from the following claim. For a sequence $\set{v_n}$ as in Theorem \ref{thm:Hom spaces equi} and for any $i\in   \set{1,r,p,rp}$ we have that
\begin{equation}
\mu_{v_n,S}^i \stackrel{\text{weak}^{*}}{\lra} \mu_{S}^i 
\mbox{ as }n\ra \infty. \label{hom equi spin}
\end{equation}
\section{Proof of Theorem \ref{thm:Hom spaces equi}}\label{sec:proof}
In this section we prove that \eqref{hom equi spin} holds. This will conclude the proof of Theorem \ref{thm:Hom spaces equi}.
\subsection{Step I: Proving \eqref{hom equi spin} for orthogonal
lattices\label{sub:Step I for lattices}}

Recall the definitions of $\overline{\bG}_{2},\overline{\bG},\cX^{S},\pi^{S}$ and $\rho$.
We define 
\[
\overline{\bL}_{v}(R)\defi\set{\pa{h,\rho(g_{v}^{-1}hg_{v})}:h\in\bH_{v}(R)}.
\]
and $\bar{\theta}_{v}=\rho(a_{v}k_{v}g_{v})$
and for $i\in \set{1,r,p,rp}$  
$$\cX^{S,i}\defi \overline{\bG}_{S}^{+}(e,g_i,e,e)\overline{\bG}(\bZ^S),\, \overline{\mu}_{S}^i\defi \mu_{\cX^{S,i}}$$ 

$$\overline{\mathbf{O}}_{v,S}^i\defi(k_{v},h_i^v,\bar\theta_{v},\rho(g_v^{-1}h_i^vg_v))\overline{\bL}_{v,S}^+\overline{\bG}(\bZ^S)\subset\cX^{S,i},\, \overline{\mu}_{v,S}^i\defi\mu_{\overline{\mathbf{O}}_{v,S}^i}$$
where again the last containment follows from (\ref{spin restriction}) and (\ref{spin real}) of Lemma~\ref{lem:spin}. 
\begin{rem}
\label{rem:projection of measures} For $i\in \set{1,r,p,rp}$ we have $\pi_{*}^{S}(\mu_{v,S}^i)=\overline{\mu}_{v,S}^i$
and $\pi_{*}^{S}(\mu_{S}^i)=\overline{\mu}_{S}^i$.\end{rem}

As a first step we prove that for any $i\in \set{1,r,p,rp}$ we have

\begin{equation}
\overline{\mu}_{v_n,S}^i \stackrel{\text{weak}^{*}}{\lra} \overline{\mu}_{S}^i,\, n\ra \infty. \label{hom equi spin lattices}
\end{equation}

Let~$C=\{v_n:n\in\bN\}$ and $S=\set{\infty,p}$ and fix $i\in \set{1,r,p,rp}$. We apply Theorem~\ref{thm:gorod-oh}
with $\mathbf{L}_{n}=\overline{\bL}_{v_{n}}$ and $g_{n}=(k_{v},h_i^v,\bar\theta_{v},\rho(g_v^{-1}h_i^vg_v))\in\overline{\bG}_{S}$.
By Lemma~\ref{lem:Isotropy of Hv } and the congruence assumption
in Theorem \ref{thm:Hom spaces equi} when $d=4\mbox{\text{ or }5}$
the main assumption to Theorem~\ref{thm:gorod-oh} is satisfied.

Let $\nu$ be a weak$^{*}$ limit of $\pa{\overline{\mu}_{v,S}^i}_{v\in C}$
which is the limit of a subsequence $\pa{\overline{\mu}_{v,S}^i}_{v\in C_{1}}$
for $C_{1}\subset C$. We wish to show that $\nu=\overline{\mu}_{S}^i$.
By Lemma  \ref{lem:Isotropy of Hv } the centralizer of $\bL_{v}$
is finite and  therefore $\bQ$-anisotropic so it follows form Theorem~\ref{thm:gorod-oh} that $\nu$ is a probability measure.

Applying Theorem~\ref{thm:gorod-oh}.(1)--(2), we find a semisimple algebraic
$\bQ$-group $\bM<\overline{\bG}$ and $C_{2}\subset C_{1}$ such
that $\av{C_{1}\setminus C_{2}}<\infty$ and for all $v\in C_{2}$
we have 
\begin{equation}
\gamma_{v}^{-1}\overline{\bL}_{v}\gamma_{v}<\bM\label{eq:subgroups of M}
\end{equation}
 for some $\gamma_{v}=\pa{\delta_v,\eta_v}\in\bG_1\times \overline{\bG}_2(\bZ^S)$.

Assume, for a moment, that $\bM=\overline{\bG}$.  We already noted above that $\overline{\mathbf{O}}_{v,S}^i\subset\cX^{S,i}$ and therefore the weak limit $\nu$ is supported inside $\cX^{S,i}$.
Using Theorem~\ref{thm:gorod-oh}.(1), we see that $\nu=\mu_{gM_{0}\overline{\bG}(\bZ^S)}$
for a subgroup $M_{0}$ of finite-index in $\overline{\bG}_{S}$ and
some $g\in\overline{\bG}_{S}$. The group $\bM_{S}^{+}=\overline{\bG}_{S}^{+}$
is a minimal finite-index subgroup of $\overline{\bG}_{S}$ (using Lemma \ref{lem:spin} minimality follows from \cite[6.7]{BT73}) and therefore
contained in $M_{0}$. Hence the orbit $gM_{0}\overline{\bG}(\bZ^S)$ contains the $\overline{\bG}_{S}^{+}$-orbit $g\overline{\bG}_{S}^{+}\overline{\bG}(\bZ^S)=\overline{\bG}_{S}^{+}g\overline{\bG}(\bZ^S)$ and is contained in the $\overline{\bG}_{S}^{+}$-orbit $\cX^{S,i}$ and therefore equal to it which shows $\nu=\overline{\mu}_{S}^i$.

Therefore, the proof of this step will be concluded once we show:
\begin{claim*}
$\bM=\overline{\bG}$.\end{claim*}
\begin{proof}[Proof of the Claim ]
Let $\pi_{1}:\bG_{1}\times\overline{\bG_{2}}\ra\bG_{1}$ and $\pi_{2}:\bG_{1}\times\overline{\bG_{2}}\ra\overline{\bG_{2}}$
denote the natural projection and define $\bM_{i}\defi\pi_{i}(\bM)$.
Since $\bM$ is semisimple and $\bG_{1}$ and $\overline{\bG_{2}}$
have non-isomorphic simple Lie factors\footnote{Because
of the ambient dimensions the accidental isomorphisms $\mathfrak s\mathfrak l_2\cong\mathfrak s\mathfrak o_3$ and $\mathfrak s\mathfrak l_4\cong\mathfrak s\mathfrak o_6$ play no role here.}, it is enough to show that
$\bM_{1}=\bG_{1}$ and $\bM_{2}=\overline{\bG}_{2}$.  

We begin with $\bM_{1}$: By \eqref{eq:subgroups of M}, $\bM_{1}$ contains subgroups
of the form $\delta_{v}^{-1}\bH_{v}\delta_{v}$ for any $v\in C_{2}$
with $\delta_{v}\in\bG_{1}(\bZ^S)$. By Lemma \ref{lem:Isotropy of Hv },
each $\delta_{v}^{-1}\bH_{v}\delta_{v}$ is a maximal semisimple subgroup.
Thus, if $\bM_{1}\neq\bG_{1}$ then for all $v,u\in C_{2}$ we have
$\delta_{v}^{-1}\bH_{v}\delta_{v}=\delta_{u}^{-1}\bH_{u}\delta_{u}$,
which says that $\bH_{\delta_{v}^{-1}v}=\bH_{\delta_{u}^{-1}u}$.
This in turn implies $\delta_{v}^{-1}v=\alpha\delta_{u}^{-1}u$ for
some $\alpha\neq0$. Fixing $v\in C_2$ we rewrite this as
\begin{equation}\label{eq:upToAlpha}\delta_{v}^{-1}v=\alpha_u\delta_{u}^{-1}u.\end{equation}  
We will show below that $\alpha_u$ is, as a real number, bounded and bounded away from $0$. Noting that for all $u\in C_{2}$ we have $\norm{\delta_{u}^{-1}u}^{2}=\norm{u}^{2}$, this will imply a contradiction as $\set{\norm{u}:u\in C_2}$ is unbounded. First note that as $v,u$ are primitive vectors in $\bZ^{d}$,
$\delta_{v}^{-1}v$ and $\delta_{u}^{-1}u$ are primitive vectors in $(\bZ^S)^{d}=\bZ[\frac{1}{p}]^{d}$
(considered as a $\bZ^S$-module). Hence $\alpha_u\in(\bZ^S)^{\times}=\set{\pm p^{n}:n\in\bZ}$. Now, from part~(3) of Theorem~\ref{thm:gorod-oh} (restricted
to the $p$-adic coordinate) there exist $g\in \bG_1(\bQ_p)$ and $h_u\in \bH_u(\bQ_p)^+$ such that $h_u^ih_u\delta_u\ra g$ for $u\in C_2$.  Thinking on \eqref{eq:upToAlpha} as equality in $\bQ_p^d$  and acting on it with $h_u^ih_u\delta_u$ we see that $\alpha_u u$ converges to $g\delta_{v}^{-1}v$ which is a fixed non-zero vector in $\bQ_p^d$. As $u$ is a primitive vector, its $p$-adic valuation is $1$ and hence the $p$-adic valuation of $\alpha_u$ is bounded, which shows that $\alpha_u$  stays bounded for all $u\in C_2$. 

Assume now that $\bM_{2}\neq\overline{\bG}_{2}$. It follows from
(\ref{eq:subgroups of M}) and from (\ref{eq:isotropy group equal to sec. factor})
that $\bM_{2}$ contains subgroups of the form $\bH_{\phi_{v}^{\eta_{v}}}$
for all $v\in C_{2}$ where $\eta_{v}\in\overline{\bG}_{2}(\bZ^S)$.
By Lemma \ref{lem:Isotropy of Hv },
$\bH_{\phi_{v}^{\eta_v}}$ is always maximal, so we have that $\bH_{\phi_{u}^{\eta_{u}}}=\bH_{\phi_{v}^{\eta_{v}}}$
for all $v,u\in C_{2}$. Fixing $v$ this implies that for each $u\in C_2$ there exist $r_u\in \bQ^\times$ such that $\phi_v\circ\eta_v=r_u(\phi_u\circ\eta_u)$ or equivalently that
\begin{equation}\label{matrices}
	\eta_{v}^{t}M_{\phi_{v}}\eta_{v} =r_u\eta_{u}^{t}M_{\phi_{u}}\eta_{u}.
\end{equation}
Applying Lemma \ref{lem:determinant and primitivity} and considering the above matrices
as primitive elements of $\operatorname{Mat}_{d-1,d-1}(\bZ[\frac1p])$,
it follows that $r_u\in\bZ[\frac1p]^\times$.
As above, we show that $r_u$ has bounded $p$-adic valuation using part~(3) of Theorem~\ref{thm:gorod-oh}. From this theorem we see that there exist  $g\in \overline\bG_2(\bQ_p)$ and $l_u\in \bH_{\phi_u}(\bQ_p)$ such that $l_u^il_u\eta_u\ra g$ for $u\in C_2$, where $l_u^i\defi\rho(g_u^{-1}h_u^ig_u)$. 
Applying the inverse of the converging sequence $l_u^il_u\eta_u$ to  \eqref{matrices}
and using again that $M_{\phi_u}\in\operatorname{Mat}_{d-1,d-1}(\bZ)$ is primitive (and
hence of $p$-adic norm one),
it follows again that the $p$-adic valuation of $r_u$ is bounded. Hence $r_u$ is, as a real number,
bounded and bounded away from $0$.
 Finally, recall that $\det M_{\phi_{u}}=\norm{u}^{2}$ and note that
$\det\eta_{v}=\det\eta_{u}=1$, so taking determinants in \eqref{matrices} we see that $\norm{v}^2=r_u^{d-1}\norm{u}^2$ which is a contradiction as $\set{\norm{u}:u\in C_2}$ is unbounded.
\end{proof}

\subsection{Step II: Upgrading from orthogonal lattices to orthogonal grids.\label{sub: ASL upgrading}}
For our fixed $i\in \set{1,r,p,rp}$ let $\nu$ be a weak$^{*}$ limit of $\pa{\mu_{v,S}^i}_{v\in C}$ which
is the limit of a subsequence $\pa{\mu_{v,S}^i}_{v\in C_{1}}$ for $C_{1}\subset C$.
We need to show that $\nu=\mu_{S}^i$. 
First notice that $\pi^{S}:\cY^{S}\ra\cX^{S}$ has compact fibers, which together with
Remark \ref{rem:projection of measures}
and $\S$\ref{sub:Step I for lattices} gives that $\pi_{*}^{S}\nu=\overline{\mu}_{S}^i$.
In particular, $\nu$ is also a probability measure.

We will use the
same type of argument as in $\S$\ref{sub:Step I for lattices}: assuming
that $\nu\neq\mu_{S}^i$ we will use the information furnished
by Theorem~\ref{thm:gorod-oh} to deduce a contradiction
to the fact that the primitive vectors in $C_{1}$ have their length
going to infinity. 

More precisely, we will apply 
Theorem~\ref{thm:gorod-oh} within the quotient 
$$
 \pa{\bG_{1}\times{\rm SL}_{d}}_{S}/\pa{\bG_{1}\times{\rm SL}_{d}}(\bZ^S).
$$
To simplify the notation we set~$\bG'={\bG_{1}\times{\rm SL}_{d}}$ and 
recall that
$$
 \bG=\bG_1\times\bG_2=\SO_{d}\times{\rm ASL}_{d-1}<\bG'.
$$ 
We note that the orbit~$\bG_S\bG'(\bZ^S)$ is isomorphic to (and will be identified with) $\cY_S=\bG_{S}/\bG(\bZ^S)$. 
Recall that this implies that the finite volume orbit~$\bG_{S} \bG'(\bZ^S)$
is a closed subset of~$\bG'_S/\bG'(\bZ^S)$,
equivalently a closed subset of~$\bG'_S$ or even of
the homogeneous space~$W=\bG_S\backslash\bG'_S$.

From Theorem~\ref{thm:gorod-oh}.(1)--(2) we find an algebraic $\bQ$-subgroup
$\bM<\bG'$ and $C_{2}\subset C_{1}$ such that
$\av{C_{1}\setminus C_{2}}<\infty$ and for all $v\in C_{2}$ 
\begin{equation}
\gamma_{v}^{-1}\bL_{v}\gamma_{v}<\bM\label{eq:subgroups of M-ASL case}
\end{equation}
for some $\gamma_{v}\in\bG'(\bZ^S)$. Moreover,
$\nu=\mu_{gM_{0}\bG'(\bZ^S)}$ for some finite-index
subgroup $M_{0}<\bM_{S}$ and $g\in \bG'_{S}$.
By construction all orbits in our sequence are contained in $\bG_{S}\bG'(\bZ^S)$,
which implies that the support of~$\nu$ is also contained in this set and in particular
that~$g\bG'(\bZ^S)\in \bG_{S}\bG'(\bZ^S)$.
This implies that we may change $\bM$ by a conjugate $\gamma\bM\gamma^{-1}$
for some $\gamma\in\bG'(\bZ^S)$ and assume
that $g\in\bG_{S}$.

If~$m\in M_0$, we obtain
that the element~$gm\bG'(\bZ^S)$ 
belongs to the support of~$\nu$. Therefore,~$\bG_S m\in W$ belongs to the closed (and discrete) set~$\bG_S\bG'(\bZ^S)$.
If~$m\in \bM_S$ is sufficiently close to the identity this implies~$m\in \bG_S$. 
As $\bM$ is Zariski connected
we conclude that $\bM<\bG$. Using the same argument it also follows from 
Theorem~\ref{thm:gorod-oh}.(3) that $\pa{\gamma_{v}}_{v\in C_{3}}\subset\bG$
for some subset~$C_3\subset C_2$ with~$|C_2\setminus C_3|<\infty$. 

By the previous step we  know that
$\pi_{*}^{S}\nu=\overline{\mu}_{S}^i$, which implies that either $\bM=\bG$
or $\bM=\bG_{1}\times\bM_{2}$ where $\bM_{2}$ is a $\bQ$-subgroup
which is $\bQ$-isomorphic to a fixed copy of ${\rm SL}{}_{d-1}$.
Such subgroups are of the form ${\rm SL}_{d-1}^{t}$ where ${\rm SL}_{d-1}^{t}$
is the conjugation of $\iota({\rm SL}{}_{d-1})=\left(\begin{array}{cc}
\begin{smallmatrix}{\rm SL}{}_{d-1}\end{smallmatrix} & 0\\
0 & 1
\end{array}\right)$ by $c_{t}=\left(\begin{array}{cc}
\begin{smallmatrix}I_{d-1}\end{smallmatrix} & t\\
0 & 1
\end{array}\right)$ for some fixed $t\in\bQ^{d-1}$. As above, we will be done once we
show that $\bM=\bG$. Assume therefore that we are in the second case and $\bM_{2}={\rm SL}_{d-1}^{t}$.
Using the definition of $\bL_{v}$ in~\eqref{def-lv} and projecting (\ref{eq:subgroups of M-ASL case})
using the canonical map $\bG\ra\bG_{2}$, we get that for all $v\in C_3$ we have
\begin{equation}\label{eq:contain}
c_{t}^{-1}\delta_v^{-1}g_{v}^{-1}\bH_{v}g_{v}\delta_vc_{t}\subset\iota({\rm SL}_{d-1}),
\end{equation}
where~$\delta_v$ denotes the projection of~$\gamma_v$ to~$\bG_2$.
Let $N\in\bN$ be such that $Nt\in\bZ^{d-1}_{\rm{prim}}$ and set $\tilde v\defi g_{v}\delta_vc_{t}(Ne_{d})$.
Note that $Ne_{d}$ is a simultaneous eigenvector of the right hand
side of \eqref{eq:contain}. It follows that for each $v\in C_{3}$, we have that $\tilde v\in(\bZ^S)^{d}$ 
and that $\tilde v$ is a simultaneous eigenvector for $\bH_{v}$. As above, since $v$ and $c_t(Ne_d)$ are primitive vectors, we get that 
$\tilde v=\alpha_{v}v$ for some $\alpha_{v}\in\bZ[\frac 1p]^\times$. As above, we show that $\alpha_v$ has bounded $p$-adic valuation using part~(3) of Theorem~\ref{thm:gorod-oh}. From this theorem we see that there exist  $g\in \bG_2(\bQ_p)$ and $l_v=g_v^{-1}h_vg_v,h_v\in \bH_{v}(\bQ_p)$ such that $l_v^il_v\delta_v\ra g$ for $v\in C_3$, where $l_v^i\defi g_v^{-1}h_v^ig_v$. Acting on $\tilde v$  with $h_v^ih_v$ we get 
$$\alpha_v v=\tilde v=h_v^ih_v\tilde v=h_v^ih_vg_{v}\delta_vc_{t}(Ne_{d})= g_v(l_v^il_v\delta_v)c_{t}(Ne_{d}).$$
Now recall that $(l_v^il_v\delta_v)c_{t}(Ne_{d})$ converges in $\bQ_p^d$
to some nonzero vector and that $g_v\in\SL_d(\bZ)\subseteq\SL_d(\bZ_p)$ does not change
the norm of $p$-adic vectors.
As $v$  has $p$-adic norm one, we see that $\alpha_v$ has bounded $p$-adic valuation, and therefore $\alpha_v$ is bounded and bounded away from $0$.

Recall the definition of $g_v$ and the vectors $v_i$ and $w$ from the introduction and note that $v_i\in v^\perp$ for $i=1,\ldots, d-1$ and $(w,v)=1$.
On the other hand, using that $\delta_v\in\bG_{2}(\bZ[\tfrac 1p])$, and the definition of $g_{v}$,
we see that 
\[
 \tilde{v}=Nw+\sum_{i=1}^{d-1}a_{i}v_{i}
 \]
  with $a_{i}\in\bZ[\tfrac 1p]$. 
 Taking the inner product of~$\tilde v$ with $v$ we get 
\[
\alpha_{v}\norm{v}^{2}=(\tilde v,v)=(Nw,v)=N
\]
for all $v\in C_3$. This gives a contradiction as $\set{\norm{v}:v\in C_3}$ is unbounded.

\section{An equivalence relation}\label{sec:equivalence relation} 

Let $G_{i}=\bG_{i}(\bR),\Gamma_{i}=\bG_{i}(\bZ)$ for $i=1,2$, $\Gamma=\Gamma_1\times\Gamma_2$, $K=\bH_{e_{d}}(\bR)$
and fix $v\in\bS^{d-1}(D)$ throughout this section. We identify $K\setminus G_{1}\cong\bS^{d-1}$
via the action of $G_{1}$ on $\bS^{d-1}$ defined by $w\mapsto k^{-1}w$ using the base point $e_d$. Indeed, this action is transitive with $K={\rm SO_{d-1}(\bR)}={\rm Stab}_{G_1}(e_{d})$. We will also write~$w.k=k^{-1}w$
for this right action of~$k\in G_1$ on~$w\in \bS^{d-1}(\bR)$.

As in the introduction, the group $K$ can also be embedded into $G_{2}$.
We denote the diagonal embedding of $K$ by $\Theta_{K}\defi\set{\pa{k,k}:k\in K}\subset G_{1}\times G_{2}$.

Let $\textbf{S}^{d-1}=\bS^{d-1}/\Gamma_{1}$ and $\textbf{S}^{d-1}(D)=\bS^{d-1}(D)/\Gamma_{1}$.
Set $\mathbf{v}=v.\Gamma_{1}$ and $[\Delta_{\mathbf{v}}]=[\Delta_{v}]$.
The latter is well-defined as $[\Delta_{\gamma v}]=[\Delta_{v}]$
for all $\gamma\in\Gamma_{1}$. Note also that the projection $\mathbf{v}\in \mathbf{S}^{d-1}(D) \mapsto\frac{\mathbf{v}}{\norm{\mathbf{v}}}\in \textbf{S}^{d-1}$
is well-defined. Recall the notation $\theta_{v}=a_{v}k_{v}g_{v}$ from Section \ref{sec:orbits and duality}. Using \eqref{eq:Delta}, it follows that the following double coset 
\begin{equation}
K\times K\pa{k_{v},\theta_{v}}\Gamma_{1}\times\Gamma_{2}\label{eq:orbit of (v,lambda v)}
\end{equation}
represents the pair 
\[
\pa{\mathbf{\frac{\mathbf{v}}{\norm{\mathbf{v}}}},[\Delta_{\mathbf{v}}]}\in\textbf{S}^{d-1}\times\cY_{d-1}.
\]
Note that all the measures appearing in equation (\ref{eq:ASL2 convergences})
are $\Gamma_{1}$-invariant. Therefore, if we consider the projection $\nu_{D}$
of $\tilde{\nu}_{D}$ to $\textbf{S}^{d-1}\times\cY_{d-1}$ we have
that the convergence (\ref{eq:ASL2 convergences}) is equivalent to
\begin{equation}
\nu_{D}\stackrel{\text{wea\ensuremath{k^{*}} }}{\lra}m_{\textbf{S}^{d-1}}\otimes m_{\cY_{d-1}}
\mbox{ as }D\ra\infty\mbox{ with }D\in A\label{eq:ASL2 conv modulu Gamm1}.
\end{equation}

\subsection{Definition of $P_{v}$ and the measure $\nu_{v}$\label{sub:Definition-of P_v}}

Set $S=\set{\infty,p}$ for some odd prime $p$
so $\bQ_{S}=\bR\times\bQ_{p}$ and $\bZ^S=\bZ[\tfrac{1}{p}]$.
For $w_1,w_2\in\bS^{d-1}(D)$ we say that $w_1\sim w_2$ if there
exist $g_{p}\in\bG_{1}(\bZ_{p})$ and $\gamma_{p}\in\bG_{1}(\bZ[\tfrac{1}{p}])$
such that $g_{p}w_1=w_2,\gamma_{p}w_1=w_2$. The equivalence relation~$\sim$ satisfies that if $w_1\sim w_2$ and $\gamma\in \Gamma_1$ then $\gamma w_1\sim\gamma w_2$, and so it 
descends to an equivalence
relation on $\textbf{S}^{d-1}(D)$. 

We set $P_{v}\defi\set{\mathbf{w}:\mathbf{w}\sim\mathbf{v}}$
and  $R_{v}=\set{\pa{\mathbf{\frac{\mathbf{\mathbf{w}}}{\norm{\mathbf{\mathbf{w}}}}},[\Delta_{\mathbf{w}}]}:\mathbf{w}\in P_{v}}$ for $v\in\bS^{d-1}(D)$.
We finally define $\nu_{v} =\nu_{D}|_{R_{v}}$. In the next section we will relate $\nu_{v}$ to the measure $\mu_{v,S}$.

\section{Deducing Theorem \ref{thm:main theorem} from Theorem \ref{thm:Hom spaces equi}.}\label{sec:translating}

\subsection{Restriction to the principal genus}
Consider the open orbit $$\cU\defi\bG(\bR\times\bZ_{p})\bG(\bZ[\tfrac{1}{p}])\subset\cY^{S}.$$
The complement of $\cU$ is a union of orbits of the open subgroup $\bG(\bR\times\bZ_{p})$ and so the set $\cU$ is also closed. 
Therefore we have that 
\begin{equation}
\eta_{v_{n}}\defi\mu_{v_{n},S}|_{\cU}\stackrel{\text{weak}^{*}}{\lra}\eta\defi\mu_S|_{\cU}\text{ as } n\to\infty\label{eq:restricted convergence}
\end{equation}
for $\set{v_{n}}$ as in Theorem \ref{thm:Hom spaces equi}.
We have a projection map $\pi$ from $\cU$, considered as a subset of
$\cY^{S}$, to $\cY^{\infty}=\bG(\bR)/\bG(\bZ)$ obtained by dividing
from the left by $\set{e}\times\bG(\bZ_{p})$. It follows that $\pi_{*}\pa{\eta}$
is a probability measure which is invariant under $\bG(\bR)$, that
is, it is the uniform probability measure on $\bG(\bR)/\bG(\bZ)$.
Therefore, under the assumptions of Theorem \ref{thm:Hom spaces equi}
we have 
\begin{equation}
\pi_{*}\pa{\eta_{v_n}}\stackrel{\text{weak}^{*}}{\lra}m_{\bG(\bR)/\bG(\bZ)}.\label{eq: equi real hom space}
\end{equation}
In addition, we have the projection map:
\begin{equation}
\rho:G_{1}\times G_{2}/\Gamma_{1}\times\Gamma_{2}\ra K\times K\setminus G_{1}\times G_{2}/\Gamma_{1}\times\Gamma_{2}\label{eq:the map p}
\end{equation}
Below we will show that the measures $\pa{\rho\circ\pi}_{*}\eta_{v}$
and $\nu_{v}$ are closely related.

\subsection{Description of $\eta_{v}$ as union of orbits}
Fix $v\in\bS^{d-1}(D)$. For $h\in\bH_{v,\set{p}}$ we set $s(h)=0$ if $h\in \bG_{1}(\bZ_{p})\bG_{1}(\bZ[\tfrac{1}{p}])$ and $s(h)=\textrm{other}$ otherwise.
We will not need this, but wish to note that
the symbol~$0$ corresponds here to our quadratic form~$Q_0(x_1,\ldots,x_d)=x_1^2+\ldots+x_d^2$ 
and~`$\textrm{other}$'
for the other quadratic forms in the genus of~$Q_0$.
 For $s\in\{0,\textrm{other}\}$
choose a set $M_{s}$ satisfying the following equality
(as subsets of $\bH_{v,\set{p}}$): 
\[
\bigcup_{h\in\bH_{v,\set{p}},s(h)=s}\bH_{v}(\bZ_{p})h\bH_{v}(\bZ[\tfrac{1}{p}])=\bigsqcup_{h\in M_{s}}\bH_{v}(\bZ_{p})h\bH_{v}(\bZ[\tfrac{1}{p}])\label{eq:h j's}
\]
where the later denotes a disjoint union. We note in passing that the sets~$M_{0}$ and  $M_{\textrm{full}}\defi M_0\cup M_{\textrm{other}}$ are finite as these double quotients are in correspondence with orbits of the open subgroup $\bH_v(\bR\times \bZ_p)$ on the \emph{compact} quotient $\bH_v(\bQ_S)/\bH_v(\bZ[\tfrac{1}{p}])$.
The finiteness of $M_0$ also follows from Proposition \ref{prop:Pv and schmidt} below.
Recall $\theta_{v}=a_{v}k_{v}g_{v}$ and note that $\pa{k_{v},\theta_{v}}\bL_{v}(\bR)=\Theta_{K}\pa{k_{v},\theta_{v}}$.
Using this we can express the orbit $\mathbf{O}_{v,S}$ from~\eqref{defOvS} in a different form: set
$l(h)=g_{v}^{-1}hg_{v}$ and let us reorder entries of products whenever convenient so that
\begin{equation}
\Theta_{K}\times\bL_{v,\set{p}}=
 \Bigl\{\pa{k,h,k,l(h)}:k\in K,h\in\bH_{v,\set{p}}\Bigr\}\subset \bG_S.\label{eq:mixing coordinates}
\end{equation}
In this notation we get 
\[
\mathbf{O}_{v,S}=\Theta_{K}\times\bL_{v,\set{p}}(k_{v},e_{p},\theta_{v},e_{p})\bG(\bZ[\tfrac{1}{p}]).
\]
We also set $\bL(\bZ_{p})=\bL_{v,\set{p}}\cap\bG(\bZ_{p})$ and obtain
\[
\mathbf{O}_{v,S}=\bigsqcup_{h\in M_{\textrm{full}}}\Theta_{K}\times\bL(\bZ_{p})(k_{v},h,\theta_{v},l(h))\bG(\bZ[\tfrac{1}{p}]),
\]
where we used the same identification as in (\ref{eq:mixing coordinates}). Furthermore the restricted measure
$\eta_{v}$ (see (\ref{eq:restricted convergence})) is a $\Theta_{K}\times\bL(\bZ_{p})$-invariant
probability measure on
\begin{equation}
\mathbf{O}_{v,S}\cap\mathcal{U}=\bigsqcup_{h\in M_{0}}\Theta_{K}\times\bL(\bZ_{p})(k_{v},h,\theta_{v},l(h))\bG(\bZ[\tfrac{1}{p}]).\label{eq:M+ union}
\end{equation}
We note that the last equality could also have been used as the definition of
the finite set~$M_0\subset\bH_{v,\set{p}}$ of representatives.

\subsection{The support of the measure $\pi_{*}(\eta_{v})$.}

By definition each $h\in M_{0}$ belongs to $\bG_{1}(\bZ_{p})\bG_{1}(\bZ[\frac{1}{p}])$.
So we can write $h=c_{1}(h)\gamma_{1}(h)^{-1}$ where $c_{1}(h)\in\bG_{1}(\bZ_{p})$
and $\gamma_{1}(h)\in\bG_{1}(\bZ[\frac{1}{p}])$. Using the fact that
$\bG_{2}$ has class number 1, we can write $l(h)=c_{2}(h)\gamma_{2}(h)^{-1}$
where $c_{2}(h)\in\bG_{2}(\bZ_{p})$ and $\gamma_{2}(h)\in\bG_{2}(\bZ[\frac{1}{p}])$. 

\begin{prop}
\label{prop:first desc of the the measure}The measure $\pi_{*}(\eta_{v})$
is a $\Theta_{K}$-invariant probability measure on \textup{
\begin{equation}
\bigsqcup_{h\in M_{0}}\cO_{h}\defi \bigsqcup_{h\in M_{0}}\Theta_{K}(k_{v}\gamma_{1}(h),a_{v}k_{v}g_{v}\gamma_{2}(h))\Gamma.\label{eq:def of Oh}
\end{equation}
}\end{prop}

\begin{proof}
The $\Theta_K$-invariance is clear from the $\Theta_K$ invariance of $\eta_v$. The proposition now follows by plugging $h=c_{1}(h)\gamma_{1}(h)^{-1}$
and $l(h)=c_{2}(h)\gamma_{2}(h)^{-1}$ into (\ref{eq:M+ union}) while
recalling two facts. The first is that the map $\pi$ is dividing
by $\set{e}\times\bG(\bZ_{p})$ from the left. The second is that
$(\gamma_{1}(h),\gamma_{1}(h),\gamma_{2}(h),\gamma_{2}(h))\in\bG(\bZ[\frac{1}{p}])$. The fact that this is a disjoint union follows from Propoisition \ref{prop:Pv and schmidt} below.
\end{proof}

Let us note that~$\Theta_{K}(k_{v}\gamma_{1}(h),a_{v}k_{v}g_{v}\gamma_{2}(h))\Gamma$ does not depend on the
choice of the representative of the double coset~$\bH_{v}(\bZ_p)h\bH_{v}(\bZ[\tfrac{1}{p}])$ and 
also not on the choice of the above decompositions. For simplicity we explain this only in the $\bG_1$ factor and the proof for both factors together is just notationally more difficult. Let us first assume 
that~$h=c_1\gamma_1^{-1}=c_1'(\gamma_1')^{-1}$ 
are two decompositions as above. This gives that
$c_{1}^{-1}c_{1}'=\gamma_{1}^{-1}\gamma_{1}'$ belongs to
$\bG_{1}(\bZ_{p})\cap\bG_{1}(\bZ[\frac{1}{p}])=\bG_{1}(\bZ)=\Gamma_1$, which implies the second half of the 
claimed independence in the~$\bG_1$-factor. 
If now~$h=c_1\gamma_1^{-1}$ as above,~$h_p\in\bH_{v}(\bZ_p)$, and~$\beta\in\bH_{v}(\bZ[\tfrac{1}{p}])$,
then~$h_ph\beta=(h_pc_1)(\gamma_1^{-1}\beta)$ and we associate to this point
the double coset $Kk_{v}\beta^{-1}\gamma_{1}\Gamma_1$. Using~$k_v\beta^{-1}k_v^{-1}\in K$
the latter equals $Kk_v\gamma_1\Gamma_1$, which is 
the claimed independence for the components in~$\bG_1$. 

We will now relate the set appearing in the above
proposition with the set~$P_v$ introduced in \S\ref{sub:Definition-of P_v}. 
\begin{prop}
\label{prop:Pv and schmidt}For $h\in M_0$ set $\varphi(h)=Kk_{v}\gamma_{1}(h)\Gamma_{1}$. Then $\varphi$ is a bijection from $M_0$ to  $\set{Kk_{u}\Gamma_{1}:\mathbf{u}\in P_{v}}$. Noting that $\varphi(h)$ corresponds to  $u=\gamma_{1}(h)^{-1}v$ we further claim that  $Ka_{v}k_{v}g_{v}\gamma_{2}(h)\Gamma_2 = [\Delta_{\mathbf{u}}]$.\end{prop}

\begin{proof}
Fix $h\in M_{0}$ and recall that $h$ stabilizes $v$. We first need to show that  $u=\gamma_{1}(h)^{-1}v\in\bZ^{d-1}$: indeed, we have 
\[
\bZ[\tfrac{1}{p}]^{d}\ni\gamma_{1}(h)^{-1}v=c_{1}(h)^{-1}hv=c_{1}(h)^{-1}v\in\bZ_{p}^{d}
\]
so $u\in\bZ^{d}$ as $\bZ[\frac{1}{p}]\cap\bZ_{p}=\bZ$. Now, the
elements $c_{1}(h),\gamma_{1}(h)$ satisfy 
\[
c_{1}(h)u=c_{1}(h)\gamma_{1}(h)^{-1}v=hv=v\mbox{ and }\gamma_{1}(h)u=v.
\]
So  $u\sim v$ and therefore
$Kk_{v}\gamma_{1}(h)\Gamma_{1}$ belongs to $\set{Kk_{u}\Gamma_{1}:\mathbf{u}\in P_{v}}$. 

To see that $\varphi$ is onto, fix $u\sim v$ and let $h_{u}=g_{p}\gamma_{p}^{-1}\in\bH_{v,\set{p}}$
arising from the definition of $\sim$ in $\S$\ref{sub:Definition-of P_v}.
Then~$\gamma_p u=v$ and $s(h_{u})=0$. Let $\bar{h}\in M_{0}$ be such that $\bH_{v}(\bZ_{p})h_{u}\bH_{v}(\bZ[\frac{1}{p}])=\bH_{v}(\bZ_{p})\bar{h}\bH_{v}(\bZ[\frac{1}{p}])$.
We have explained above that $Kk_{u}\Gamma_{1}=Kk_{v}\gamma_{p}\Gamma_{1}=Kk_{v}\gamma_{1}(\bar{h})\Gamma_{1}$.

For injectivity, let  $h_1,h_2\in M_0$ and set $\alpha_i=\gamma_1(h_i),k_i=c_1(h_i),\, i=1,2$. Assuming $\varphi(h_1)= \varphi(h_2)$, there exists a $\gamma\in \Gamma_1$ such that $Kk_v\alpha_1\gamma=Kk_v\alpha_2$. Thus $\alpha_1\gamma \alpha_2^{-1}$ stabilizes $v$ so  $\alpha_1\gamma \alpha_2^{-1}\in\bH_v(\bR)\cap\bG_1(\bZ[\tfrac{1}{p}])=\bH_v(\bZ[\tfrac{1}{p}])$. Also $(k_2\gamma^{-1} k_1^{-1})v=(h_2\alpha_2\gamma^{-1} \alpha_1^{-1}h_1^{-1})v=v$ so $k_2\gamma^{-1} k_1^{-1}\in  \bH_v(\bQ_p)\cap \bG_1(\bZ_p)= \bH_v(\bZ_p)$. As $(k_2\gamma^{-1} k_1^{-1})h_1(\alpha_1\gamma \alpha_2^{-1})=h_2$ we see that  
 $$\bH_{v}(\bZ_{p})h_1\bH_{v}(\bZ[\tfrac{1}{p}])= \bH_{v}(\bZ_{p})h_2\bH_{v}(\bZ[\tfrac{1}{p}]).$$ 

For the second assertion, fix $h\in M_{0}$ and let $u=\gamma_{1}(h)^{-1}v$. 
We will use the abbreviations $\gamma_{i}=\gamma_{i}(h),c_{i}=c_{i}(h)$ for $i=1,2$
which satisfy by definition that~$h=c_1\gamma_1^{-1}$ and~$l(h)=g_v^{-1}hg_v=c_2\gamma_2^{-1}$.
We need to show that 
\begin{equation}
Ka_{v}k_{v}g_{v}\gamma_{2}\Gamma_{2}\stackrel{?}{=}[\Delta_{u}]=Ka_{u}k_{u}g_{u}\Gamma_{2}.\label{eq:Moved coset-1}
\end{equation}
Note first that $a_{v}=a_{u}$ and that $k_{v}\gamma_{1}$ is a legitimate
choice of $k_{u}$.  With these choices (and using
the identity of~$K$ on both sides), (\ref{eq:Moved coset-1})
will follow once we show $g_{u}^{-1}\gamma_{1}^{-1}g_{v}\gamma_{2}\in\Gamma_{2}$.
The element $g_{u}^{-1}\gamma_{1}^{-1}g_{v}\gamma_{2}$ is certainly
a determinant 1 element which maps $\bR^{d-1}$ to itself. Furthermore,
the last entry of its last column is positive by the orientation
requirement in the definition of $g_{v}$ and $g_{u}$. Therefore,
it will be enough to show that this element maps $\bZ^{d}$ to itself.
Using $\bZ[\frac{1}{p}]\cap\bZ_{p}=\bZ$ again, this follows from 
\[
\begin{aligned}\bZ[\tfrac{1}{p}]^{d}\supset g_{u}^{-1}\gamma_{1}^{-1}g_{v}\gamma_{2}\bZ^{d}=g_{u}^{-1}c_{1}^{-1}\pa{c_{1}\gamma_{1}^{-1}}g_{v} & \pa{\gamma_{2}c_{2}^{-1}}c_{2}\bZ^{d}=\\
=g_{u}^{-1}c_{1}^{-1}hg_{v}g_{v}^{-1}{h}^{-1}g_{v}c_{2}\bZ^{d} & =g_{u}^{-1}c_{1}^{-1}g_{v}c_{2}\bZ^{d}\subset\bZ_{p}^{d}.
\end{aligned}
\]
\end{proof}
\subsection{Weights of $\pa{\rho\circ\pi}_{*}\eta_{v}$  and $\nu_{v}$}
Fix a sequence $\pa{v_n}$ of vectors satisfying the conditions of Theorem~\ref{thm:Hom spaces equi} and set 
$\mu_n\defi\pa{\rho\circ\pi}_{*}\eta_{v_n}$ (with $\pi$ as in \eqref{eq: equi real hom space} and $\rho$ as in \eqref{eq:the map p}) and $\nu_n\defi \nu_{v_n}$.
It follows from Propositions \ref{prop:first desc of the the measure}--\ref{prop:Pv and schmidt} that $R_{v_n}=\rm Supp(\nu_n)=\rm Supp(\mu_n)$. Let $\lambda_n$ denote the normalised counting measure on $R_{v_n}$. In this section we show 
\begin{equation}
 \mu_n-\lambda_n \slra{n\ra\infty} 0 \text{ and } \nu_n-\lambda_n\slra{n\ra\infty} 0, \label{eq:close to counting}
\end{equation}
That is, the measures $\mu_n$  and $\nu_n$ are equal to $\lambda_n$  up-to a negligible error. For $\bu\in\mb{S}^{d-1}$ let $S(\mathbf{u})=\av{\on{Stab}_{\Gamma_1}(u)}$ for some $u\in \mathbf{u}$ and $E=\tilde{E}\times \cY_{d-1}$ where
$$\tilde{E}\defi \set{\bu\in\mb{S}^{d-1}: S(\mathbf{u})>1}\subset \mb{S}^{d-1}.$$
The convergences in (\ref{eq:close to counting}) follow from the following two lemmata:
\begin{lem}\label{lem:weights}
Fix $n\in \bN$ and let $v=v_n$. Set $M_n=\max_{x\in R_{v}}\mu_n(x)$ and $N_n=\max_{x\in R_{v}}\nu_n(x)$ and $a=\av{\Gamma_1}$. For every $x\in R_{v}$,   $\frac{M_n}{a}\leq \mu_n(x)\leq M_n$ and $\frac{N_n}{a} \leq\nu_n(x)\leq N_n$. Furthermore, equality holds on the right hand side of both inequalities when $x\in R_v\setminus E$.
\end{lem}
\begin{lem}\label{lem:negligible E}
We have that 
\begin{equation}
 \frac{|R_{v_n}\cap E|}{|R_{v_n}|}\rightarrow 0\mbox{ as }n\ra \infty.
 \label{eq:QcapE is negligible}
\end{equation}
\end{lem}
\begin{proof}[Proof of Lemma  \ref{lem:weights}]
For $x\in R_{v_n}$ let $S(x)\defi S(\mathbf{\frac{\mathbf{\mathbf{u}}}{\norm{\mathbf{\mathbf{u}}}}})$ where $x=x(u)=\upa$. By the definition of $\nu_n$, we have for $x(u)\in R_{v_n}$ that  $$\nu_n(x(u))=\frac{\av{\Gamma_1}/S(x(u))}{\sum_{y\in R_{v_n}}\av{\Gamma_1}/S(y)}=\frac{S(x(u))^{-1}}{\sum_{y\in R_{v_n}}S(y)^{-1}}.$$ So the lemma follows for $\nu_n$. For $\mu_n$ first note that, using  (\ref{eq:M+ union}) we have $$\mu_n(x(u))=\eta_{v_n}(\Theta_{K}\times\bL_v(\bZ_{p})(k_{v},h,\theta_{v},l(h))\bG(\bZ[\tfrac{1}{p}]))$$ where $h=h(x(u))$ is the unique (by Prop. \ref{prop:Pv and schmidt}) element corresponding to $x(u)$ in $M_0$. 
Therefore, we will be done once we show that the stabilizer  of the above orbit, namely, 

\begin{equation}
\av{\pa{\Theta_{K}\times\bL_v(\bZ_{p})}\cap\alpha_{h}\bG(\bZ[\tfrac{1}{p}]))\alpha_{h}^{-1}}\label{eq:u stabilizer}
\end{equation}
is bounded by $S(x(u))$, where $\alpha_h\defi (k_{v},h,\theta_{v},l(h))$. To this end, notice that as $\Theta_{K}\times\bL_v(\bZ_{p})$ embeds diagonally into the product space $\bG_S$, the third and the fourth coordinate of an element in this stabilizer are determined by the first and the second. As we are only interested in getting an upper bound it is enough to consider the stabilizer  in $\bG_1$.  Using that $Kk_v=k_v\bH_v(\bR)$ and  $\bH_v(\bZ_{p})\subset \bG_1(\bZ_{p})$ it is enough to bound 
$$ \av{(\bH_v(\bR)\times \bG_1(\bZ_{p}))\cap(e,h)\bG_1(\bZ[\tfrac{1}{p}])(e,h^{-1})}.$$
Using the decomposition $h=c\gamma \defi c(h)\gamma(h)^{-1}$ and that $c \in  \bG_1(\bZ_{p})$ the latter is bounded by 
\begin{equation}
\av{(\bH_v(\bR)\times \gamma\bG_1(\bZ_{p})\gamma^{-1})\cap\bG_1(\bZ[\tfrac{1}{p}])}.\label{eq:last stab}
\end{equation}
As $\gamma\in \bG_1(\bZ[\tfrac{1}{p}])$ we have $\gamma\bG_1(\bZ_{p})\gamma^{-1}\cap \bG_1(\bZ[\tfrac{1}{p}])=\gamma\bG_1(\bZ)\gamma^{-1}$. Therefore $$\av{\bH_v(\bR)\cap\gamma\bG_1(\bZ)\gamma^{-1}}=\av{\gamma^{-1}\bH_v(\bR)\gamma\cap\bG_1(\bZ)}=S(x(u))$$ bounds (\ref{eq:last stab}).
\end{proof}
\begin{proof}[Proof of Lemma \ref{lem:negligible E} ]
We have that $\mu_n(E)=\mu_n(\tilde E\times \cY_{d-1})\slra{n\ra\infty} 0$ since by (\ref{eq: equi real hom space}) we have  $\limsup_{n\ra\infty}\pa{\pi_1}_{*}\mu_n(\tilde E)\leq m_{\mb{S}^{d-1}}(\tilde E)=0$. Here $\pi_1:\mb{S}^{d-1}\times \cY_{d-1}\ra \mb{S}^{d-1}$ is the projection map. Using Lemma \ref{lem:weights} we have
\begin{equation}
\mu_n(E)=\frac{\mu_n(E\cap R_{v_n})}{\mu_n(R_{v_n})}
\ge \frac{\frac{M_n}{a}|E\cap R_{v_n}|}
{M_n|R_{v_n}|}
\ge 
\frac{1}{a}\frac{|E\cap R_{v_n}|}{|R_{v_n}|}
\end{equation}
which gives~\eqref{eq:QcapE is negligible}.

\end{proof}
This shows \eqref{eq:close to counting} and thus that 
\begin{equation}
\lim_{n\ra \infty} \nu_n=\lim_{n\ra \infty}\mu_n=m_{\mathbf{S}^{d-1}}\otimes m_{\cY_{d-1}}. \label{eq:same limit}
\end{equation}

\subsection{Concluding the proof of Theorem \ref{thm:main theorem}.}

We have to show that the convergence in (\ref{eq:ASL2 conv modulu Gamm1})
holds. In fact, we have proven a stronger statement. The support of
$\nu_{D}$ can be written as a disjoint union of equivalence classes
of the form $R_{v}$ for some $v\in\bS^{d-1}(D)$. The convergence
in \eqref{eq:same limit} shows that \emph{each }sequence of the form $(\nu_{v})$ for any
choice of varying vectors $v$ (under the congruence condition $\norm{v}^{2}\in\bD(p)$
when $d=4\text{ or }5$), equidistribute to $m_{\textbf{S}^{d-1}}\otimes m_{\cY_{d-1}}$.
This implies Theorem \ref{thm:main theorem}.

\author{\bibliographystyle{plain}
\bibliography{SOS}
}

\end{document}